\documentclass[a4paper,12pt]{article}
\usepackage{latexsym}
\usepackage{amssymb}
\usepackage{amsfonts}
\usepackage{amsmath}
\usepackage{amsthm}

\newenvironment{@abssec}[1]{%
    \if@twocolumn

      \section*{#1}%
    \else

      \vspace{.05in}\footnotesize
      \parindent .2in
 {\upshape\bfseries #1. }\ignorespaces
    \fi}

    {\if@twocolumn\else\par\vspace{.1in}\fi}

\newenvironment{keywords}{\begin{@abssec}{\keywordsname}}{\end{@abssec}}

\newenvironment{AMS}{\begin{@abssec}{\AMSname}}{\end{@abssec}}

\newcommand\keywordsname{Key words}
\newcommand\AMSname{AMS subject classifications}
\newcommand\AMname{AMS subject classification}
\newtheorem{theorem}{Theorem}
 \newtheorem{lemma}[theorem]{Lemma}

\def\qed{\vbox{\hrule height0.6pt\hbox{%
  \vrule height1.3ex width0.6pt\hskip0.8ex
  \vrule width0.6pt}\hrule height0.6pt
 }}

\def\theequation{\arabic{section}.\arabic{equation}}

\def\theequation{\arabic{section}.\arabic{equation}}
 \def\thetheorem{\arabic{section}.\arabic{theorem}}

\def\theequation{\arabic{section}.\arabic{equation}}
 \def\thetheorem{\arabic{section}.\arabic{theorem}}

\def\veps{\varepsilon}

\def\RE{\mathbb R}
\def\ovr{\overline}

\def\la{\lambda}
\def\ka{\kappa}
\def\pa{\partial}
\def\na{\nabla}

\def\eps{\varepsilon}

\def\Om{\Omega}

\def\dist{\mbox{\rm dist}}

\title{Nonlinear diffusion with a bounded \\
stationary  level surface\thanks{This research was partially supported by Grants-in-Aid
for Scientific Research (B) ($\sharp$ 15340047 and $\sharp$ 20340031) and a Grant-in-Aid for Exploratory Research ($\sharp$ 18654027) of
Japan Society for the Promotion of Science, and
by a
Grant of the Ital\-ian MURST.}}

\author{Rolando Magnanini\thanks{Dipartimento di Matematica U.~Dini,
Universit\` a di Firenze, viale Morgagni 67/A, 50134 Firenze, Italy.
({\tt magnanin@math.unifi.it}).}
  \quad and\quad Shigeru Sakaguchi\thanks{Department of Applied Mathematics,
Graduate School of  Engineering, Hiroshima
University, Higashi-Hiroshima, 739-8527,  Japan.
({\tt sakaguch@amath.hiroshima-u.ac.jp}).}}
\begin{document}

\maketitle

\begin{abstract}
We consider nonlinear diffusion of some substance in a container   (not necessarily bounded) with bounded boundary 
of class $C^2$. 
Suppose that, initially,  the container is empty
and, at all times, the substance at its boundary is kept
at density $1.$ We show that, if the container contains
a proper $C^2$-subdomain on whose boundary the substance has constant density at each given time,
then the boundary of the container must be a sphere. We also consider  nonlinear diffusion in the whole $\mathbb R^N$
of some substance whose density is initially a characteristic function of the complement of a domain with bounded $C^2$ boundary, and obtain similar results. These results are also extended to the heat flow in the sphere $\mathbb S^N$ and
the hyperbolic space $\mathbb H^N.$

\vskip 2ex
\centerline{\bf R\'esum\'e}
\vskip 1ex

Nous consid\'erons la diffusion non lin\'eaire d'une substance dans un r\'ecipient (pas n\'ecessairement born\'e) avec 
fronti\`ere born\'ee de classe $ C^2 $. 
Supposons qu'initialement, le r\'ecipient soit vide et, \`a sa fronti\`ere, la densit\'e de la substance soit 
gard\'ee \`a tout moment \'egale \`a $1.$  Nous montrons que, si le r\'ecipient contient un sous-domaine $C^2$ propre
\`a la fronti\`ere duquel la substance est gard\'ee \`a tout moment \`a densit\'e constante, alors la fronti\`ere du r\'ecipient doit \^etre une sph\`ere. Nous consid\'erons aussi la diffusion non lin\'eaire dans tout $\mathbb R^N $ d'une substance  dont la densit\'e est initialement une fonction caract\'eristique du compl\'ementaire d'un domaine ayant la fronti\`ere born\'e et $ C^2$,  et nous obtenons des r\'esultats semblables. 
Ces r\'esultats sont aussi g\'en\'eralis\'es au 
cas du flux de chaleur dans la sph\`ere $\mathbb  S^N $ et l'espace hyperbolique $ \mathbb  H^N. $
\end{abstract}



\pagestyle{plain}
\thispagestyle{plain}
\markboth{R. MAGNANINI AND S. SAKAGUCHI}{Nonlinear Diffusion}


\section{Introduction}
\subsection{Background}
\label{subs:1}

 In the paper \cite{MS3}, we considered
 the solution $u=u(x,t)$ of the following initial-boundary value problem for the heat equation:
\begin{eqnarray}
&u_t=\Delta u\ \ &\mbox{in }\ \Omega\times (0,+\infty),\label{heat}\\
&u=1\ \ &\mbox{on }\ \pa\Omega\times (0,+\infty),\label{dirichlet}\\
&u=0\ \ &\mbox{on }\ \Omega\times \{0\},\label{initial}
\end{eqnarray}
where $\Omega$ is a bounded domain in $\mathbb R^N$ with $N \ge 2$, and we obtained
the following symmetry result.


\renewcommand{\thetheorem}{A}

\begin{theorem}{\rm (\cite{MS3})}
\label{th:matzoh}
Let $\Omega$ be a bounded domain in $\mathbb R^N, \ N\ge 2,$
satisfying the exterior sphere condition
and suppose that $D$ is a domain, with boundary $\pa D,$
satisfying the interior cone condition,
and such that $\overline{D}\subset\Omega.$
\par
Assume that the solution
$u$ of problem {\rm (\ref{heat})-(\ref{initial}) } 
is such that
\begin{equation}
\label{isotherm}
u(x,t)=a(t),\ \ (x,t)\in \pa D\times (0,+\infty),
\end{equation}
for some function $a: (0,+\infty)\to (0,+\infty).$
Then $\Omega$ must be a ball.
\end{theorem}

\renewcommand{\thetheorem}{\arabic{section}.\arabic{theorem}}
 \setcounter{theorem}{0}

A surface satisfying \eqref{isotherm} is said to be a {\it stationary isothermic surface.} 
We recall that $\Omega$ satisfies the {\it exterior sphere condition}
if for every $y\in\pa\Omega$ there exists a ball $B_r(z)$ such that
$\overline{B_r(z)}\cap\overline{\Omega}=\{y\},$ where $B_r(z)$ denotes
an open ball centered at $z \in \RE^N$ and with radius $r > 0.$ Also,
$D$ satisfies the
{\it interior cone condition} if for every $x \in \pa D$ there exists a
finite right spherical cone $K_x$
with vertex $x$ such that $K_x \subset \overline{D}$ and $\overline{K_x}
\cap \pa D = \{x\}.$

In order to better understand the background of the present paper, 
we outline  the proof of Theorem \ref{th:matzoh} improved by a result in
\cite{MS4}. The proof is essentially based on three ingredients.

The first one is a result of Varadhan's \cite{Va} which states
that, as $t\to 0^+,$ the function
$-4t\log u(x,t)$ converges uniformly
on $\overline\Omega$ to the function $d(x)^2$, where
\begin{equation*}
\label{distance}
d(x)=\dist(x,\pa\Omega),\ \ x\in\Omega.
\end{equation*}
To apply this result one needs the boundary $\partial\Omega$ be also the boundary of the exterior $\mathbb R^N \setminus \overline{\Omega}$. The assumption that $\Omega$ satisfies the exterior sphere condition is sufficient 
for that to happen. Hence, by (\ref{isotherm}) there exists $R > 0$ satisfying
\begin{equation}
\label{parallel}
d(x) = R\ \mbox{ for every }\ x \in \partial D.
\end{equation}

The second ingredient consists of a {\it balance law} proved in \cite{MS1} and \cite{MS2} (see \cite{MS3} for another proof). It states that, in any domain $G$ in $\RE^N,$ a solution $v=v(x,t)$ 
of the heat equation is zero at some point $x_0\in G$ {\it for every} $t > 0$ if and only if
\begin{equation}
\label{eq:secondbalance}
\int\limits_{\partial B_r(x_0)}v(x,t)\ dS_x=0,\
\mbox{for every } r\in [0,\dist(x_0,\partial G)) \mbox{ and } t > 0.
\end{equation}
\par
We use \eqref{eq:secondbalance} in two different ways. In the former one,
we choose $G=\Om$ and $v=u_{x_i},$ $i=1,\dots, N,$ and obtain, by some manipulations, that
the gradient $\na u$ is zero at some point $x_0\in\Om$ for every $t>0$ if and only if 
\begin{equation}
\label{eq:balance}
\int\limits_{\pa B_r(x_0)}(x-x_0) u(x,t)\ dS_x=0,\
\mbox{for every } r\in [0,d(x_0)) \mbox{ and } t > 0.
\end{equation}
This condition helps us show that both $\pa D$ and $\pa\Om$ must be analytic.
Indeed, with the aid of the interior cone condition for $D$, 
by combining (\ref{eq:balance}) and (\ref{parallel}) with the short-time behavior of $u$ 
described in Varadhan \cite{Va}, we can see that for every point $x_0 \in \partial D$ there exists a time $t_0>0$ satisfying $\nabla u(x_0,t_0) \not=0$; this implies that $\partial D$ is analytic. 
Thus, by using the exterior sphere condition for $\Omega$ again, we can conclude  that $\partial\Omega$ is analytic and parallel to $\partial D$. 
\par
In the latter way of using \eqref{eq:secondbalance}, we choose two distinct points
$P, Q \in \partial\Omega$ and let $p, q \in \partial D$ be the points such that
$$
\overline{B_R(p)}\cap\partial\Omega = \{P\}\ \mbox{ and }\ \overline{B_R(q)}\cap\partial\Omega = \{Q\}.
$$
Thence, we consider the function $v = v(x,t)$ defined by
$$
v(x,t) = u(x+p,t) - u(x+q,t)\ \mbox{ for }\ (x,t) \in B_R(0) \times (0,+\infty).
$$
Since $v$ satisfies the heat equation and $v(0,t) = a(t)-a(t)= 0$ for every $t > 0$, it follows from (\ref{eq:secondbalance}) that 
$$
t^{-\frac {N+1}4} \int_{B_R(p)} u(x,t)\ dx = t^{-\frac {N+1}4} \int_{B_R(q)} u(x,t)\ dx\ \mbox{ for every }\ t > 0.
$$
Therefore, by taking advantage of the boundary layer for $u$ for short times, we let $t \to 0^+$ and
by using a result in \cite{MS4}, we obtain that
\begin{equation}
\label{prodcurv}
C(N)\left\{\prod\limits_{j=1}^{N-1}\left[\frac1R -\ka_j(P)\right]\right\}^{-\frac 12}=C(N)\left\{\prod\limits_{j=1}^{N-1}\left[\frac1R -\ka_j(Q)\right]\right\}^{-\frac 12},
\end{equation}
where $\ka_j(x), j=1,\dots, N-1,$ denotes the $j$-th principal
curvature of the surface $\pa\Omega$ at the point $x\in\pa\Omega$ with respect to the inward unit normal vector to $\partial\Omega$, and where $C(N)$ is a positive constant depending only on $N$ (see \cite{MS4}, Theorem 4.2).
\par
With \eqref{prodcurv} in hand, we are ready to use our third ingredient:
Aleksandrov's sphere theorem \cite[p. 412]{Alek}. (A special case of this theorem is the well-known {\it Soap-Bubble Theorem} (see also \cite{Reil}).) Since \eqref{prodcurv} implies that
$\prod\limits_{j=1}^{N-1}\left[\frac1R -\ka_j(x)\right]$ is constant for $x \in \partial\Omega$, 
by applying Aleksandrov's sphere theorem, we conclude that $\partial\Omega$ must be a sphere (see \cite{MS3} and \cite{MS4} for details).

\subsection{Main results}

In the present paper, we extend and improve the results described in subsection \ref{subs:1} 
to the case of certain {\it nonlinear} diffusion equations. It is evident that 
the introduction of a nonlinearity immediately rules out the use of our second ingredient,
e.g. the balance law. 
\par
Since this was crucial to prove the necessary regularity of $\pa\Om,$
we will have to change our assumptions on the domain $\Om.$ 
Thus, we shall assume $\Omega$ to be a domain  (not necessarily bounded) in $\mathbb R^N, N \ge 2,$ 
having {\it bounded} boundary of class $C^2,$ that is, $\pa\Om$ consists of $m\ (m \ge 1)$ 
connected components $S_1, \cdots, S_m \subset \partial\Omega$  which are the boundaries of  
bounded $C^2$-domains $G^1, \cdots. G^m$ in $\mathbb R^N$, respectively. Thus 
\begin{equation}
\label{domains we consider}
\partial\Omega = \bigcup_{j=1}^m S_j \ \mbox{ and }\ S_j = \partial G^j \mbox{ for each } j \in \{ 1, \cdots, m\}.
\end{equation}
\par
It should also be noticed that the lack of a balance law precludes the proof
of property \eqref{prodcurv} (unless we find an alternative proof) and hence
Aleksandrov's sphere theorem cannot be put in action.
We shall overcome this difficulty by a new and more direct proof of symmetry 
only based on our first ingredient (conveniently modified in Theorem \ref{th:varadhan-nonlinear}) and 
Serrin's {\it method of moving planes} (see \cite{Ser}, \cite{Reic}, \cite{Sir}). It is worth mentioning that our proof 
does not need  Serrin's {\it corner lemma} but simply uses the strong maximum principle
and Hopf boundary lemma
(see Theorems \ref{th:matzoh-nonlinear} and \ref{th:matzoh-nonlinear-unbounded}).
\par
We now set up our framework. We consider the unique bounded solution 
$u=u(x,t)$ of the nonlinear diffusion equation
\begin{eqnarray}
u_t =\Delta \phi(u)\quad\mbox{ in }\ &\Omega\times (0,+\infty),\label{nonlinear-eq}
\end{eqnarray}
subject to conditions \eqref{dirichlet} and \eqref{initial}. Here $\phi: \mathbb R\to \mathbb R$ is such that 
\begin{eqnarray}
&&\phi\in  C^2(\mathbb R), \quad \phi(0) =0, \mbox{ and } \label{smoothness}
\\
&&0 < \delta_1 \le \phi^\prime(s) \le \delta_2 \ \mbox{ for } s \in \mathbb R,\label{uniformly-parabolic}
\end{eqnarray}
where $\delta_1, \ \delta_2$ are positive constants.
By the maximum principle we know that
\begin{equation*}
\label{bounded}
0 < u < 1\ \mbox{ in }\ \Omega \times (0,+\infty).
\end{equation*}
\par
Let $\Phi = \Phi(s)$ be the function defined by
\begin{equation}
\label{large-phi}
\Phi(s) = \int_1^s \frac {\phi^\prime(\xi)}\xi\ d\xi\ \mbox{ for }\ s > 0.
\end{equation}
Note that if $\phi(s) \equiv s$, then $\Phi(s) = \log s$. 
\par
We extend Varadhan's result to our setting by the following


\begin{theorem}
\label{th:varadhan-nonlinear} 
Let $u$ be the solution of problem \eqref{nonlinear-eq}, \eqref{dirichlet}-\eqref{initial}.
\par
Then, 
$$
\lim_{t\to 0^+} -4t\Phi(u(x,t))=d(x)^2
$$
uniformly on every compact set in $\Om.$
\end{theorem}
\par
The proof of this theorem is constructed by adapting well-known results 
of the theory of viscosity solutions (\cite{CrIL}, \cite{Koi}, \cite{EI}, \cite{ES}, \cite{LSV}).
The techniques developed to prove Theorem \ref{th:varadhan-nonlinear} 
can be used to extend this result to the important case in which the
homogeneous boundary condition \eqref{dirichlet} is replaced by 
the non-homogeneous one
\begin{eqnarray}
u=f \quad\quad\ \ \ \ \mbox{ on } &\partial\Omega\times
(0,+\infty),\label{nonlinear-Dirichlet-2}
\end{eqnarray}
where $f = f(x)$ is a continuous function on $\pa\Om,$ 
bounded from above and away from zero by positive constants 
(see Theorem \ref{th:varadhan-nonlinear-2}).
\par
The following symmetry result corresponds
to Theorem \ref{th:matzoh} and Theorem 3.1 in \cite{MS5}.

\begin{theorem}
\label{th:matzoh-nonlinear} 
Let $D$ be a  $C^2$ domain in $\mathbb R^N$ satisfying $\overline{D} \subset \Omega$. 
Assume that the solution
$u$ of problem \eqref{nonlinear-eq}, \eqref{dirichlet}-\eqref{initial}, satisfies
\eqref{isotherm}.
\par
Then $m=1$ and $\partial\Omega$ must be a sphere.
\end{theorem}
\par
When $\Omega$ is limited to unbounded domains, we have
\begin{theorem}
\label{th:matzoh-nonlinear-unbounded} 
Let $D$ be a  $C^2$ unbounded domain in $\mathbb R^N$ satisfying $\overline{D} \subset \Omega$. 
\par
Assume that, for any connected component $\Gamma$ of $\partial D$, the solution
$u$ of problem \eqref{nonlinear-eq}, \eqref{dirichlet}-\eqref{initial}, satisfies
the following condition:
\begin{equation}
\label{level-invariant}
u(x,t)=a_\Gamma(t),\ \ (x,t)\in \Gamma\times (0,+\infty),
\end{equation}
for some function $a_\Gamma: (0,+\infty)\to (0,+\infty).$
\par
Then $m=1$ and $\partial\Omega$ must be a sphere.
\end{theorem}
\par
When $\phi(s)=s$ and $\Omega$ is bounded, Theorem \ref{th:matzoh} is clearly stronger than Theorem \ref{th:matzoh-nonlinear}, since
in the former we can use the balance law to infer better regularity. 
Furthermore, the same techniques used for the proof of Theorem \ref{th:matzoh} also yield
a more general version of it (see Theorem \ref{th:matzoh-annulus}).
\par
The paper is then organized as follows. In Section \ref{section2}, we prove 
all our symmetry results: Theorems \ref{th:matzoh-nonlinear}, \ref{th:matzoh-nonlinear-unbounded}
and \ref{th:matzoh-annulus}. In Section \ref{section3}, with the help of the theory of viscosity solutions, we prove Theorem  \ref{th:varadhan-nonlinear} and its extension, Theorem \ref{th:varadhan-nonlinear-2}. 
Section \ref{section4} is devoted to show similar results for the unique bounded solution of the  Cauchy problem for nonlinear diffusion equations. In Section \ref{section5}, we mention that this kind of results also hold 
for the heat flow in the sphere $\mathbb S^N$ and the hyperbolic space $\mathbb H^N$ with $N \ge 2$.


\setcounter{equation}{0}
\setcounter{theorem}{0}

\section{Symmetry results.}
\label{section2}
In this section,  with the aid of Theorem \ref{th:varadhan-nonlinear},  by applying the method of moving planes to problem \eqref{nonlinear-eq}, \eqref{dirichlet}-\eqref{initial} directly, we prove Theorems \ref{th:matzoh-nonlinear} and \ref{th:matzoh-nonlinear-unbounded}. 

\begin{proofA}
 First of all, we consider the case where $\Omega$ is unbounded. 
In this case $\Omega$ is an exterior domain, that is, we have 
$$
\overline{G^i} \cap \overline{G^j} = \emptyset\ \mbox{ if }\ i \not=j, \, i, j=1,\dots, m, \ \mbox{ and }\ \Omega = \mathbb R^N \setminus \left(\bigcup_{j=1}^m\overline{G^j}\right)
$$
(see (\ref{domains we consider}) for the definitions of $S_j$ and $G^j$). Set
\begin{equation}
\label{domain outside}
G = \bigcup_{j=1}^m G^j.
\end{equation}
Then $G$ is a bounded open set in $\mathbb R^N$ having $m$ connected components $G^1, \cdots G^m$.
Theorem \ref{th:varadhan-nonlinear} and the assumption (\ref{isotherm}) yield (\ref{parallel}). 
Furthermore, with the aid of our $C^2$-smoothness assumption on $\partial D$ and $\partial\Omega$, we see that  
both $\pa\Om$ and $\pa D$ consist of $m$ connected closed hypersurfaces and each component of $\pa\Om$ 
is parallel, at distance $R,$ to only one component of $\pa D.$
 
We apply the method of moving planes to the open set $G$. 
The proof  runs similarly to those of Serrin's \cite{Ser} --- or Reichel's \cite{Reic} and Sirakov's \cite{Sir} 
for exterior domains --- but with the major difference that, here, 
since the relevant overdetermination takes place {\it inside} $\Om,$ Serrin's corner lemma --- an extension of Hopf
boundary lemma to domains with corners --- is not needed.  
\par
Let $\ell$ be a unit vector in $\mathbb R^N,$ $\lambda\in\mathbb R,$ and let $\pi_\lambda$ be the hyperplane $x\cdot\ell=\lambda.$
For large $\lambda,$ $\pi_\lambda$ will be disjoint from $\overline{G};$ as $\lambda$ decreases, 
$\pi_\lambda$ will intersect $\overline{G}$ and cut off from $G$ an open cap $G_\lambda$
(on the same side of $\lambda\to +\infty$). 
\par
Denote by $G_\lambda^\prime$ the reflection
of $G_\lambda$ in the plane $\pi_\lambda.$ At the beginning, $G_\lambda^\prime$ will be 
and remain in $G$ until one of the following occurs:
\begin{enumerate}
\item[(i)] $G_\lambda^\prime$ becomes internally tangent to $\partial G$
at some point $P$ not on $\pi_\lambda;$
\item[(ii)] $\pi_\lambda$ reaches a position in which it is orthogonal to $\partial G$
at some point $Q.$ 
\end{enumerate}
\par
Let $\lambda_*$ denote the (minimal) value of $\lambda$ at which the plane $\pi_\lambda$ reaches one of these positions and
suppose that $G$ is not symmetric with respect to $\pi_{\lambda_*}$. Let  
$\Omega_\ell$ be the connected component of $\Omega\cap\{ x \in \mathbb R^N : x\cdot\ell < \lambda_*\}$ whose boundary contains the points $P$ or $Q$ in
the respective cases (i) or (ii). Since, as already observed, $\partial\Omega$ 
and $\partial D$ consist of connected closed 
pairwise parallel hypersurfaces, we can find points  $P^*$ and $Q^*$ in $\partial D$ such that
$|P-P^*|$ or $|Q-Q^*|$ equal $R$, respectively, and we have that $P^* \in \Omega_\ell$ and $Q^* \in \partial\Omega_\ell\cap\pi_{\lambda_*}.$
\par
Let $x^\lambda=x+2[\lambda -(x\cdot\ell)]\ell$ denote the reflection
of a point $x\in\RE^N$ in the plane $\pi_\lambda.$ 
For $(x,t)\in\Omega_\ell\times(0,\infty),$ consider the function $w=w(x,t)$ defined by
$$
w(x,t) = u(x^{\lambda_*},t).
$$
Then it follows from \eqref{isotherm} that 
\begin{equation}
\label{iso-level}
w(P^*,t) = u(P^*,t)\ \mbox{ or }\ \frac {\partial u}{\partial \ell} (Q^*,t) = 0\ \mbox{ for all } t >0,
\end{equation}
where in the second equality we have used the fact that the vector $\ell$ is tangential also to $\partial D$ at $Q^*\in\partial D$.
\par
Observe that $w$ and $u$ satisfy
\begin{eqnarray*}
&w_t = \Delta \phi(w)\ \mbox{ and } u_t = \Delta \phi(u) \ &\mbox{ in }\ \Omega_\ell \times (0,+\infty),\label{2 porous equations}\\
&w=u\ \ &\mbox{ on }\ \left(\pa\Omega_\ell\cap \pi_{\lambda_*}\right)\times (0,+\infty),\label{on the plane}\\
& w < 1=u\ \ &\mbox{ on }\  \left(\pa\Omega_\ell\setminus \pi_{\lambda_*}\right)\times (0,+\infty), \label{not on the plane}\\
&w=u=0\ \ &\mbox{ on }\ \Omega_\ell\times \{0\}.\label{2 initial data}
\end{eqnarray*}
Hence, by the strong comparison principle, 
\begin{equation}
\label{strict inequality}
w < u\ \mbox{ in }\ \Omega_\ell \times (0,+\infty).
\end{equation}
Indeed, \eqref{strict inequality} can be obtained by applying the strong comparison principle to the bounded solutions 
$W =\phi(w)$ and $U = \phi(u)$ of $W_t =\frac 1{\psi^\prime(W)}\,\Delta W$ and $U_t =\frac 1{\psi^\prime(U)}\,\Delta U$, 
respectively; here, $\psi$ is the inverse function of $ \phi$. 
\par
If case (i) applies,  \eqref{strict inequality} contradicts the first equality in (\ref{iso-level}),
since $P^*\in\Omega_\ell.$ 
If case (ii) applies,  by using Hopf's boundary point lemma, we can infer that 
 $$
 \frac {\partial u}{\partial \ell} (Q^*,t) < 0\ \mbox{ for all } t >0,
 $$
 which contradicts the second equality in (\ref{iso-level}).
\par
 In conclusion, $G$ is symmetric for any direction $\ell \in \mathbb R^N$, and in view of the definition (\ref{domain outside}) of $G$, $m = 1$ and $G$ must be a ball. Namely, $\Omega$ is the exterior of a ball and $\partial\Omega$ must be a sphere.
\par
When $\Omega$ is bounded, it suffices to apply the method of moving planes directly to $\Omega$.
\end{proofA}
\vskip.2cm
\begin{proofB}
With the aid of the $C^2$ smoothness assumption of both $\partial D$ and $\partial\Omega$, Theorem \ref{th:varadhan-nonlinear} and the assumption \eqref{level-invariant}, together with the fact that $D$ is unbounded,  yield that $\partial\Omega$ and $\partial D$ consist of $m$ pairs of connected closed hypersurfaces being parallel to each other respectively.
(When $D$ is bounded, $\partial D$ may consist of two connected components being parallel to one component  of $\partial\Omega$. )
Hence, the proof runs similarly to that of Theorem \ref{th:matzoh-nonlinear},
with the only difference that
the components in each pair constituting $\pa\Om\cup\pa D$ may be at different distance from one another. 
\end{proofB}
\vskip.2cm
We conclude this section with a more general version of
Theorem \ref{th:matzoh}.

\begin{theorem}
\label{th:matzoh-annulus}
Let $\Omega$ be a  domain (not necessarily bounded) in $\mathbb R^N, \ N\ge 2,$
satisfying the exterior sphere condition
and suppose that $\partial\Omega$ is bounded.
Let $D$ be a domain with $\overline{D}\subset\Omega$, and let $\Gamma$ be a connected component of $\partial D$ satisfying
\begin{equation}
\label{nearest condition}
\dist(\Gamma, \partial\Omega) = \dist(\partial D,\partial \Omega).
\end{equation}
Suppose that $D$ satisfies the interior cone condition on $\Gamma$.
Assume that the solution
$u$ of problem {\rm (\ref{heat})-(\ref{initial})} satisfies
\eqref{level-invariant}.
\par
Then $\partial\Omega$ must be either a sphere or the union of two concentric spheres.
\end{theorem}
\begin{proof}
Because of the assumption (\ref{nearest condition}), the proofs of Lemma 2.2 of \cite{MS5}
and Lemma 3.1 in \cite{MS3} also work in this situation. 
Then, there exists a connected component $S$ of $\partial\Omega$ such that both $\Gamma$ and $S$ 
are  analytic and these are parallel with distance $R = \dist\left(\Gamma,\partial\Omega\right);$  also, $\prod\limits_{j=1}^{N-1}\left[\frac1R -\kappa_j(x)\right]$ is constant for $x \in S$. 
Since $S$ is bounded, by applying  Aleksandrov's sphere theorem \cite{Alek} to this equation, we see that $S$ and $\Gamma$ are concentric spheres. 
\par
Let $E$ be the annulus with $\partial E = S\cup\Gamma$. With the help of the analyticity of $u$, 
by proceeding as in the proof of Theorem 3.1 in \cite{MS5}, we see that for any  $i \not=j$
$$
-(x_j-a_j)\frac{\partial u(x,t)}{\partial x_i} + (x_i-a_i)\frac{\partial u(x,t)}{\partial x_j} = 0\quad\mbox{ in }\ \Omega \times (0,+\infty),
$$
where the point $a = (a_1, \dots, a_N) \in \mathbb R^N$ is the center of the sphere $S$. 
Hence $u$ must be radially symmetric with respect to $a.$ 
\end{proof}


\setcounter{equation}{0}
\setcounter{theorem}{0}

\section{Short-time behavior of solutions of nonlinear diffusion equations}
\label{section3}
In this section, with the help of the theory of viscosity solutions, we prove 
our keystone result, Theorem \ref{th:varadhan-nonlinear}. We begin with some preliminaries.
\begin{lemma}
\label{le:comparison} 
Let $w=\phi(u),$ where $u$ is the solution of \eqref{nonlinear-eq}, \eqref{dirichlet}-\eqref{initial}.
For $j = 1, 2,$ let $w_j$ solve the problem:
\begin{eqnarray}
& (w_j)_t=\delta_j\Delta w_j\ \ &\mbox{ in }\ \Omega\times (0,+\infty),\label{delta-heat}\\
&w_j= \phi(1)\ \ &\mbox{ on }\ \pa\Omega\times (0,+\infty),\label{delta-dirichlet}\\
&w_j=0\ \ &\mbox{ on }\ \Omega\times \{0\}.\label{delta-initial}
\end{eqnarray}
\par
Then
$$
w_1 \le w \le w_2\ \mbox{ in }\ \Omega \times (0,+\infty).
$$
\end{lemma}
\begin{proof}
Since $w_t=\phi^\prime(u)\Delta w,$ by \eqref{uniformly-parabolic} we have:
\begin{equation}
\label{compare}
\delta_1\,\Delta w \le w_t \le \delta_2\,\Delta w\ \mbox{ in }\ \Omega \times (0,+\infty).
\end{equation}
Hence, by the comparison principle we get our claim. 
\end{proof}
\vskip.2cm

Now, let $\Psi = \Psi(s)$ be the inverse function of $\Phi$. Then 
$$
s = \Phi(\Psi(s)) = \int_1^{\Psi(s)} \frac {\phi^\prime(\xi)}\xi\ d\xi
$$
and
\begin{equation}
\label{inverse-diff}
\Psi(s) = \phi^\prime(\Psi(s)) \Psi^\prime(s),
\end{equation}
by differentiating in $s.$
\par
As in Freidlin and Wentzell \cite{FW}, for $0 < \varepsilon < 1$, define the function $u^\varepsilon= u^\varepsilon(x,t)$ by
$$
u^\varepsilon(x,t) = u(x,\varepsilon t)\ \mbox{ for }\ (x,t) \in \Omega \times (0,+\infty).
$$
Then $u^\varepsilon$ satisfies
\begin{eqnarray*}
&u^\varepsilon_t =\varepsilon\Delta \phi(u^\varepsilon)\quad\mbox{ in }\ &\Omega\times (0,+\infty),\label{nonlinear-eq-var}\\
&u^\varepsilon=1 \quad\quad\ \ \ \ \ \mbox{ on } &\partial\Omega\times
(0,+\infty),\label{nonlinear-Dirichlet-var}\\
&u^\varepsilon=0 \quad\quad\ \ \ \ \ \mbox{ on } &\Omega\times
\{0\}.\label{nonlinear-initial-var}
\end{eqnarray*}
Moreover, the function $v^\varepsilon = v^\varepsilon(x,t)$ defined by
$$
v^\varepsilon(x,t) = -\varepsilon\Phi(u^\varepsilon(x,t))\ \mbox{ for }\ (x,t) \in \Omega \times (0,+\infty).
$$
is such that $u^\varepsilon = \Psi\left(-\varepsilon^{-1} v^\varepsilon\right)$ 
and, by \eqref{inverse-diff}, we have that
\begin{eqnarray}
&v^\varepsilon_t = \varepsilon\,\phi^\prime\,\Delta v^\varepsilon - |\nabla v^\varepsilon|^2\quad\ \mbox{ in }\  &\Omega\times (0,+\infty),\label{pressure-equation}\\
&v^\varepsilon=0 \quad\quad\qquad\qquad\qquad\ \mbox{ on } &\partial\Omega\times (0,+\infty),\label{Dirichlet-pressure}\\
&v^\varepsilon=+\infty \qquad\qquad\quad\quad\ \ \ \mbox{ on } &\Omega\times
\{0\},\label{initial-pressure}
\end{eqnarray}
where $\phi^\prime= \phi^\prime\left(\Psi\left(-\varepsilon^{-1}v^\varepsilon\right)\right)$.

\begin{lemma}
\label{le:uniform-estimate} 
It holds that for $(x,t) \in \overline{\Omega} \times (0,+\infty)$
$$
\frac{\delta_1}{\delta_2}\cdot\frac 1{4t} d(x)^2 \le \liminf_{\varepsilon \to 0^+} v^\varepsilon(x,t) \le  \limsup_{\varepsilon \to 0^+} v^\varepsilon(x,t)\le \frac{\delta_2}{\delta_1}\cdot\frac 1{4t} d(x)^2,
$$
where these limits as $\varepsilon \to 0^+$ are uniform in every compact set contained in $\overline{\Omega} \times (0, +\infty)$.
\end{lemma}

\begin{proof}
We observe that the following hold:
\begin{eqnarray}
& \delta_1 s \le \phi(s) \le \delta_2 s\ &\mbox{ for }\ s \ge 0,\label{ine:first}
\\
&-\delta_1 \log s \le -\Phi(s) \le -\delta_2\log s\ &\mbox{ for }\ 0 < s \le 1,\label{ine:second}
\\
&e^{s/\delta_1}
\le \Psi(s) \le e^{s/\delta_2}
\ &\mbox{ for }\ -\infty < s \le 0.\label{ine:third}
\end{eqnarray}
Let $w_j^\varepsilon=w_j^\varepsilon(x,t)\,  (j=1, 2)$ be the functions defined by
$$
w_j^\varepsilon(x,t) = w_j(x,\varepsilon t),
$$
where the $w_j$'s are defined in Lemma \ref{le:comparison}.
With the aid of (\ref{ine:first}) and (\ref{ine:second}), it follows from Lemma \ref{le:comparison} that
\begin{equation*}
\label{estimates-above-below}
-\varepsilon\delta_1\log\left(\frac{w_2^\varepsilon}{\delta_1}\right) \le v^\varepsilon \le -\varepsilon\delta_2\log\left(\frac{w_1^\varepsilon}{\delta_2}\right)\ \mbox{ in }\ \Omega \times (0,+\infty).
\end{equation*}
\par
By the result of Varadhan \cite{Va}, we see that, as $\varepsilon \to 0^+,$ the functions $-\varepsilon\delta_j\log w_j^\varepsilon$  converge to the function $\frac1{4t} d(x)^2$ uniformly
on every compact set contained in $\overline{\Omega} \times (0,+\infty)$, since each scaled function $\frac 1{\phi(1)}w_j(x, \delta_j^{-1}t)$ solves problem (\ref{heat})-(\ref{initial}). 
Our claim then follows at once.
\end{proof}
\vskip.2cm
The next lemma easily follows from Lemma \ref{le:uniform-estimate}.

\begin{lemma}
\label{le:uniform-bound} 
For any compact set $K$ in $\Omega \times(0,+\infty)$, there exist three positive 
constants $\varepsilon_0, c_1,$ and $c_2$ $(0 < c_1\le c_2)$
depending on $K$  such that
$$
0 < c_1 \le v^\varepsilon \le c_2\ \mbox{ in }\ K,
$$
for $0 < \eps \le \eps_0.$ 
\end{lemma}
The key point in the proof of Theorem \ref{th:varadhan-nonlinear} is to obtain 
the following gradient estimate which we shall prove at the end of this section.
\begin{lemma}
\label{le:gradient-bound} 
For any compact set $K$ in $\Omega \times(0,+\infty)$, there exist two positive constants $\eps_1$ 
$(\eps_1\le\eps_0)$ and $c_3$ depending on $K,$ such that 
$$
|\nabla v^\varepsilon| \le c_3\ \mbox{ in }\ K,
$$
for $0 < \eps \le \eps_1.$
\end{lemma}
Then, by combining Lemmas \ref{le:uniform-bound} and \ref{le:gradient-bound} with Gilding's result \cite{Gild}, 
we obtain the following uniform H\" older estimate.
\begin{lemma}
\label{le:holder-bound-in-time} 
For any compact set $K$ in $\Omega \times(0,+\infty)$, there exist two positive constants $\eps_2$ $(\eps_2\le\eps_1)$
and $c_4$ depending on $K,$ such that 
$$
|v^\varepsilon(x,t)-v^\varepsilon(x,s)| \le c_4 |t-s|^{\frac 12}\ \mbox{ for any pair}\ (x,t), (x,s) \in K
$$
and for $0 < \eps \le \eps_2.$
\end{lemma}

 \begin{theorem}
\label{th:convergence} 
The following limit
$$
\lim_{\veps\to 0^+} v^\varepsilon(x,t)=\frac 1{4t}\, d(x)^2
$$
holds uniformly on every compact set in $\Omega\times(0,+\infty)$.
\end{theorem}
\begin{proof}
Lemmas \ref{le:uniform-bound}, \ref{le:gradient-bound}, and \ref{le:holder-bound-in-time} 
together with Ascoli-Arzel\` a's theorem and the Cantor diagonal process yield  
a positive vanishing sequence of numbers $\eps_n$ and a continuous function $v = v(x,t)$ in $\Omega\times(0,+\infty)$ such that, as $n \to \infty$, the $v^{\eps_n}$'s converge to $v$ uniformly on every compact set contained in $\Omega \times(0,+\infty).$ Hence, by  Lemma \ref{le:uniform-estimate}, 
\begin{equation}
\label{estimate for limit function}
\frac{\delta_1}{\delta_2}\cdot\frac 1{4t} d(x)^2 \le v(x,t)\le \frac{\delta_2}{\delta_1}\cdot\frac 1{4t} d(x)^2\ \mbox{ for } (x,t)\in\Omega \times (0,+\infty).
\end{equation}
\par
Define a function $V = V(x,t)$ on $\mathbb R^N \times (0,+\infty)$ by
\begin{equation*}
\label{extended solution}
V(x,t) = \left\{\begin{array}{rll}
 v(x,t)\ &\mbox{ if }\ x \in \Omega,
\\
0\ &\mbox{ if  }\ x  \not\in \Omega.
\end{array}\right.
\end{equation*}
Since both $d^2$ and its gradient vanish on $\partial\Omega$,  (\ref{estimate for limit function}) 
yields that $V$ is continuous on $\mathbb R^N \times (0,+\infty)$, differentiable 
at any point on $\partial\Omega \times (0,+\infty)$, and that 
both $V $ and $\nabla V$ vanish on $\partial\Omega \times (0,+\infty)$. 
Also, (\ref{estimate for limit function}) yields that $\lim\limits_{t \to 0^+} v(x,t) = +\infty$ if $x \in \Omega$. \par
Therefore, by using the fact that $v^\varepsilon$ solves problem (\ref{pressure-equation})-(\ref{initial-pressure}), with the help of Crandall, Ishii, and Lions \cite{CrIL},  we see that  $V$ is a viscosity solution of the following Cauchy problem:
\begin{eqnarray}
&V_t = - |\nabla V|^2\quad\mbox{ in } &\mathbb R^N\times (0,+\infty),\nonumber\\
&V=0 \quad\qquad\ \mbox{ on } &(\mathbb R^N\setminus\Omega)\times \{0\},\label{eiconal-equation}\\
&V=+\infty \quad\ \ \ \mbox{ on } &\Omega\times
\{0\}.\nonumber
\end{eqnarray}
Moreover,  since a uniqueness result of Str\"omberg \cite{Str} tells us  that the Hopf-Lax formula provides  the unique viscosity solution of the Cauchy problem \eqref{eiconal-equation}, we must have that, for any $(x,t) \in \mathbb R^N \times (0,+\infty)$,
$$
V(x,t) = \inf\left\{ \varphi(\xi) + \frac {|x-\xi|^2}{4t}\ :\ \xi \in \mathbb R^N\ \right\} = \frac {\left(\dist(x, \mathbb R^N \setminus \Omega) \right)^2}{4t},
$$
where $\varphi =\varphi(\xi)$ is the lower semicontinuous initial data defined by 
$$
\varphi(\xi) =\left\{\begin{array}{rll}
 +\infty\ &\mbox{ if }\ \xi \in \Omega,
\\
0\ &\mbox{ if  }\ \xi  \not\in \Omega.
\end{array}\right.
$$ 
\par 
By the uniqueness of $V$,  the whole sequence $\{v^\eps\}$ converges as $\eps \to 0^+$, and we get 
our claim.
\end{proof}
\vskip.2cm
\begin{proofD}
The  desired result follows by simply setting $t=1$ and then $\varepsilon = t$ in Theorem \ref{th:convergence}. 
\end{proofD}
\vskip.2cm
By a simple argument, we can extend Theorem \ref{th:varadhan-nonlinear} to
the important case of non-homogeneous boundary values.
\begin{theorem}
\label{th:varadhan-nonlinear-2} 
Let $f=f(x)$ be a continuous function on $\pa\Om$ such that
 \begin{equation}
 \label{bounded from above and below}
 0 < b_1 \le f(x) \le b_2\ \mbox{ for all } x \in \partial\Omega,
 \end{equation}
 for some positive constants $b_1$ and $b_2.$ 
Let $u$ be the solution of problem \eqref{nonlinear-eq}, \eqref{nonlinear-Dirichlet-2}, \eqref{initial}.
\par
Then, 
$$
\lim_{t\to 0^+} -4t\Phi(u(x,t))=d(x)^2
$$
uniformly on every compact set in $\Om.$
\end{theorem}

\begin{proof}
Consider the unique bounded solutions $u^j=u^j(x,t)\ (j=1,2)$ of the following initial-boundary value problems:
\begin{eqnarray*}
&u^j_t =\Delta \phi(u^j)\quad&\mbox{ in }\ \Omega\times (0,+\infty),\label{nonlinear-eq-j}\\
&u^j= b_j\quad &\mbox{ on }\ \partial\Omega\times
(0,+\infty),\label{nonlinear-Dirichlet-j}\\
&u^j=0 \quad & \mbox{ on }\ \Omega\times
\{0\}.\label{nonlinear-initial-j}
\end{eqnarray*}
Then it follows from (\ref{bounded from above and below}) and the comparison principle that
\begin{equation}
\label{upper and lower solutions  j}
u^1 \le u \le u^2\ \mbox{ in }\ \Omega\times (0,+\infty).
\end{equation}
\par
With the help of Theorem \ref{th:varadhan-nonlinear}, we see that, 
as $t\to 0^+,$ the function 
$-4t\Phi(u^j(x,t))$ converges to the function $d(x)^2$ uniformly
on every compact set in $\Omega$ for each $j=1,2$.
Indeed, for each $j =1,2,$ we set
$$
U = \frac {u^j}{b_j},\ \tilde{\phi}(s) = \frac 1{b_j} \phi(b_j s)\ \mbox{ for } s \in \mathbb R,\ \mbox{ and }\ \tilde{\Phi}(s) = \int_1^s\frac {\tilde{\phi}^\prime(\xi)}\xi\ d\xi\  \mbox{ for } s > 0.
$$
Then  it follows that 
\begin{equation}
\label{relationship between Phi and tildePhi}
\tilde{\phi}^\prime(s) = \phi^\prime(b_js)\ \mbox{ for } s \in \mathbb R,\ \tilde{\Phi}(s) = \Phi(b_js)-\Phi(b_j)\ \mbox{ for } s >0,
\end{equation}
and 
\begin{eqnarray}
&U_t =\Delta \tilde{\phi}(U)\quad&\mbox{ in }\ \Omega\times (0,+\infty),\label{nonlinear-eq-m}\\
&U= 1\quad&\mbox{ on }\ \partial\Omega\times
(0,+\infty),\label{nonlinear-Dirichlet-m}\\
&U=0 \quad& \mbox{ on }\ \Omega\times
\{0\}.\label{nonlinear-initial-m}
\end{eqnarray}
Thus,  applying Theorem \ref{th:varadhan-nonlinear} to $U$ yields that, as $t\to 0^+,$ the function
$-4t\tilde{\Phi}(U(x,t))$ converges to the function $d(x)^2$ uniformly
on every compact set in $\Omega$. Hence, with the aid of the second equality of (\ref{relationship between Phi and tildePhi}), this means that, 
as $t\to 0^+,$ the function
$-4t\Phi(u^j(x,t))$ converges to the function $d(x)^2$ uniformly
on every compact set in $\Omega$.

  On the other hand, since $\Phi$ is increasing in $s >0$, we have from (\ref{upper and lower solutions j}) that
$$
-4t \Phi(u^1) \ge -4t\Phi(u) \ge -4t\Phi(u^2)\ \mbox{ in }\ \Omega\times (0,+\infty),
$$
which implies that, as $t\to 0^+,$ the function
$-4t\Phi(u(x,t))$ converges to the function $d(x)^2$ uniformly
on every compact set in $\Omega$. 
\end{proof}
\vskip.2cm
\begin{proofC}
We use Bernstein's technique (see \cite{EI}, \cite{Koi}, \cite{ES}, and \cite{LSV}). Let $r, \tau$ and $T$
be positive numbers such that $\tau<2\tau<T$ and $K \subset B_r(0) \times [2\tau,T].$  
Take $\zeta \in C^\infty(B_{2r}(0) \times (\tau, T])$ satisfying
\begin{eqnarray*}
&& 0 \le \zeta\le 1 \ \mbox{ and } \ \zeta_t \ge 0\ \mbox{ in } B_{2r}(0) \times (\tau, T],
\\
&&\zeta = 1\ \mbox{ on }\ B_r(0) \times [2\tau,T],
\mbox{ and } \ \mbox{ supp } \zeta \subset B_{2r}(0) \times (\tau,T].
\end{eqnarray*}
In the sequel of this proof, we will use the constants $\eps_0, c_1$ and $c_2$ of Lemma
\ref{le:uniform-bound} relative to the compact set
$\ovr{B_{2r}(0)}\times [\tau, T].$
\par
Consider the function $z = z(x,t)$ defined by
\begin{equation}
\label{auxiliary-function}
z = \zeta^2|\nabla v^\varepsilon|^2-\lambda v^\varepsilon,
\end{equation}
where $\lambda > 0$ is a constant to be determined later, and $0 < \eps \le \eps_0.$
Suppose that $(x_0,t_0)$ is a point in $B_{2r}(0) \times (\tau,T]$ satisfying
$$
\zeta(x_0,t_0) > 0\ \mbox{ and }\ \max_{\overline{B_{2r}(0)} \times [\tau, T]} z = z(x_0,t_0).
$$
At $(x_0,t_0)$ we then have
\begin{equation}
\label{at-maximum}
z_t\ge0,\ z_{x_i}=0,\ \mbox{ and }\ \Delta z \le 0,
\end{equation}
and hence
\begin{equation*}
\label{parabolic-inequality}
0 \le z_t - \varepsilon\phi^\prime\left(\Psi(-\varepsilon^{-1}v^\varepsilon)\right)\Delta z.
\end{equation*}
\par
The following inequality holds at $(x_0,t_0)$ for some positive 
constants $A_1$ and $A_2$ independent on $(x_0,t_0)$ and $\eps:$
\begin{equation}
\label{main-inequality}
\lambda |\nabla v^\varepsilon|^2 \le A_1|\nabla v^\varepsilon|^2+A_2\zeta|\nabla v^\varepsilon|^3
-2\zeta^2|\nabla v^\varepsilon|^2\phi^{\prime\prime}\Psi^\prime\Delta v^\varepsilon - \varepsilon\phi^\prime\zeta^2|\nabla^2 v^\varepsilon|^2.
\end{equation}
It is a consequence of \eqref{at-maximum}
 and some lengthy calculations
that, for the reader's convenience, will be carried out in the Appendix A.
\par
Now, we want to bound the third and fourth addendum on the right-hand side of \eqref{main-inequality}.
The bound for the latter addendum,
$$
- \varepsilon\phi^\prime\zeta^2|\nabla^2 v^\varepsilon|^2 \le - \varepsilon\delta_1\zeta^2|\nabla^2 v^\varepsilon|^2,
$$
easily follows from \eqref{uniformly-parabolic}. 
In order to bound the former one, we use the fact that
$\phi\in C^2(\RE)$ and Lemma \ref{le:uniform-bound}, the algebraic inequality $2ab\le a^2+b^2,$ 
and the key inequality
\begin{equation}
\label{most-important}
\displaystyle
0 < \Psi^\prime\left(-\varepsilon^{-1}v^\varepsilon\right) = \frac {\Psi\left(-\varepsilon^{-1}v^\varepsilon\right)}{\phi^\prime} \le \frac 1{\delta_1}e^{-\frac {v^\varepsilon}{\varepsilon\delta_2}} \le \frac 1{\delta_1}e^{-\frac {c_1}{\varepsilon\delta_2}},
\end{equation}
which follows from \eqref{inverse-diff}, \eqref{uniformly-parabolic} and \eqref{ine:third}.
With these three ingredients, we show that 
\begin{eqnarray*}
-2\zeta^2|\nabla v^\varepsilon|^2\phi^{\prime\prime}\Psi^\prime\Delta v^\varepsilon \le 
\frac 1{\delta_1}e^{-\frac {c_1}{\varepsilon\delta_2}}\,\zeta^2( A_3|\nabla v^\varepsilon|^4 + |\nabla^2 v^\varepsilon|^2),
\end{eqnarray*}
at $(x_0,t_0),$ for some positive constant $A_3$ independent of $(x_0,t_0)$ and $\eps.$ 
\par
Set
$$
M = \max_{\overline{B_{2r}(0)} \times [\tau, T]} \zeta|\nabla v^\varepsilon|, \ \ \lambda =  \frac {M^2+1}{2(c_2+1)},
$$
and choose $\varepsilon_*$ in $(0,\varepsilon_0]$ so small to obtain that
$$
\frac{A_3}{\delta_1}e^{-\frac {c_1}{\varepsilon\delta_2}}\le \frac 1{4(c_2+1)}\ \mbox{ and }\ 
\frac 1{\delta_1}e^{-\frac {c_1}{\varepsilon\delta_2}} \le \varepsilon\delta_1
$$
for all $\varepsilon \in (0,\varepsilon_*].$ Then, with these choices of constants,
from (\ref{main-inequality}) and the aforementioned
bounds on the second-order derivatives of $v^\eps,$ we have that
\begin{equation}
\label{second-main-inequality}
\frac {M^2+1}{4(c_2+1)}|\nabla v^\varepsilon|^2 \le A_1|\nabla v^\varepsilon|^2 + A_2M|\nabla v^\varepsilon|^2
\end{equation}
at $(x_0,t_0),$ for any $\varepsilon \in (0,\varepsilon_*].$
\par 
Thus, if $\nabla v^\varepsilon(x_0,t_0) \not=0,$ from (\ref{second-main-inequality}) we get
$$
\frac {M^2+1}{4(c_2+1)} \le A_1 + A_2M,
$$
which yields the desired gradient estimate at once. 
If $\nabla v^\varepsilon(x_0,t_0)=0$, instead, we use 
the definition \eqref{auxiliary-function} of $z$ to infer that
$$
M^2 \le \max z+ \lambda \max v^\varepsilon \le \lambda \max v^\varepsilon  \le \frac {M^2+1}{2(c_2+1)}\, c_2 \le \frac {M^2}2 + \frac 12,
$$
since $z(x_0,t_0)=-\la v^\eps(x_0,t_0)<0.$ Therefore, $M \le 1$ and this completes the proof.
\end{proofC}
\vskip2ex
\noindent
{\bf Remark.} \ Lions, Souganidis, and Vazquez \cite{LSV} consider the pressure equation for the porous medium equation:
$$
(v_m)_t= (m-1)v_m\Delta v_m + |\nabla v_m|^2\ \mbox{ for }\ m > 1,
$$
and consider the asymptotic behavior as $m \to 1^+$.  They get the interior gradient estimate for $v_m$ independent on $m$ by a technique similar to ours. We follow the outline of their proof but we use inequality (\ref{most-important}) in order to overcome the difficulty caused by $\phi^\prime= \phi^\prime\left(\Psi\left(-\varepsilon^{-1}v^\varepsilon\right)\right)$ in equation (\ref{pressure-equation}).


\setcounter{equation}{0}
\setcounter{theorem}{0}

\section{On the Cauchy problem}
\label{section4}

Let $\Omega$ be a domain given in (\ref{domains we consider}) and consider the unique bounded solution $u=u(x,t)$ of the following Cauchy problem:
\begin{equation}
\label{Cauchy}
u_t = \Delta \phi(u)\ \mbox{ in } \mathbb R^N \times (0,+\infty),\ \mbox{ and }\ u = \chi_{\mathbb R^N\setminus\Omega}\ \mbox{ on }\mathbb R^N \times\{0\},
\end{equation}
where $ \chi_{\mathbb R^N\setminus\Omega}$ denotes the characteristic function of the set $\mathbb R^N \setminus\Omega$ and $\phi$ satisfies the assumptions \eqref{smoothness}-\eqref{uniformly-parabolic}. The purpose of this section is to prove the following result.
\begin{theorem}
\label{th:Cauchy} 
Theorems  {\rm \ref{th:varadhan-nonlinear}}, {\rm\ref{th:matzoh-nonlinear}}, and {\rm\ref{th:matzoh-nonlinear-unbounded}} also hold for the unique bounded solution $u$ of the Cauchy problem {\rm(\ref{Cauchy})}.
\end{theorem}

\vskip.2cm
Let us start with two lemmas.
\begin{lemma}
\label{le: a boundary value problem for ODE} 
There exist a small  $\delta > 0$ and a $C^2$-function $f = f(\xi)$ on $\mathbb R$ satisfying
\begin{eqnarray*}
&& (\phi^\prime(f)f^\prime)^\prime + \frac 12(\xi+2\delta)f^\prime = 0\ \mbox{ and }\ f^\prime < 0\ \mbox{ in }\ \mathbb R, \label{2nd-order ODE}
\\
&& \mbox{ and }\ 1 > f(-\infty) > f(0) > 0 >  f(+\infty) > -\infty.\label{boundary values}
\end{eqnarray*}
\end{lemma}

\noindent
{\it Proof. }  It suffices to show that there exists  a $C^2$-function $h= h(\xi)$ on $\mathbb R$ satisfying
\begin{eqnarray}
&& (\phi^\prime(h)h^\prime)^\prime + \frac 12 \xi\ h^\prime = 0\ \mbox{ and }\ h^\prime < 0\ \mbox{ in }\ \mathbb R, \label{2nd-order ODE for h}
\\
&& \mbox{ and }\ 1 > h(-\infty) > h(0) > 0 >  h(+\infty) > -\infty.\label{boundary values for h}
\end{eqnarray}
Indeed, setting $f(\xi) = h(\xi + 2\delta)$ for sufficiently small $\delta >0$ gives the desired solution $f.$ 
\par
The assumptions \eqref{smoothness}-\eqref{uniformly-parabolic} guarantee existence and uniqueness, 
{\it on the whole $\RE,$} of the solution $(h,H)$ of the Cauchy problem for the system of ordinary differential equations 
\begin{equation}
\label{system of ODE}
h^\prime = \frac H{\phi^\prime(h)},\ H^\prime = -\frac 12 \xi  \frac H{\phi^\prime(h)}, \mbox{ and } (h(0),H(0)) = (h_0,H_0),
\end{equation}
(obtained by letting $H = \phi^\prime(h)h^\prime$ in \eqref{2nd-order ODE for h}); here $h_0 >0$ and $H_0 < 0$ are given numbers. Also, by uniqueness we infer that $H< 0$ on $\mathbb R$  and hence $h^\prime < 0$ on $\mathbb R$. 
\par
Thus, with the help of (\ref{uniformly-parabolic}), by integrating the second equation
in \eqref{system of ODE}, we have that
$$
H_0\exp\left\{ -\frac{\xi^2}{4\delta_2}\right\} \le H(\xi) \le  H_0\exp\left\{ -\frac{\xi^2}{4\delta_1}\right\} <0
$$
and 
$$
\frac{H_0}{\delta_1}\exp\left\{ -\frac{\xi^2}{4\delta_2}\right\}\le h'(\xi)\le
\frac{H_0}{\delta_2}\exp\left\{ -\frac{\xi^2}{4\delta_1}\right\}<0
$$
for $\xi\in\RE;$ hence  
$$
h_0+ \frac{H_0}{\delta_1}\int_0^\xi\exp\left\{ -\frac{\eta^2}{4\delta_2}\right\} d\eta \le h(\xi)\le h_0 +\frac{H_0}{\delta_2}\int_0^\xi \exp\left\{ -\frac{\eta^2}{4\delta_1}\right\} d\eta,
$$
for $\xi > 0,$ and 
$$
h_0 +\frac{H_0}{\delta_2}\int_0^\xi \exp\left\{ -\frac{\eta^2}{4\delta_1}\right\} d\eta\le h(\xi) \le h_0+ \frac{H_0}{\delta_1}\int_0^\xi\exp\left\{ -\frac{\eta^2}{4\delta_2}\right\} d\eta.
$$
for $\xi < 0.$
By letting $\xi \to +\infty$ and $\xi \to -\infty$, respectively, we get
\begin{eqnarray*}
&& h_0 + \frac {H_0}{\delta_1}\sqrt{\pi\delta_2} \le h(+\infty) \le h_0+ \frac {H_0}{\delta_2}\sqrt{\pi\delta_1},\label{limit at +infinity}
\\
&&h_0 - \frac {H_0}{\delta_2}\sqrt{\pi\delta_1} \le h(-\infty) \le h_0-\frac {H_0}{\delta_1}\sqrt{\pi\delta_2}.\label{limit at -infinity}
\end{eqnarray*}
\par
Therefore, \eqref{boundary values for h} is obtained by setting
$$
h_0 = \frac {\delta_1^{3/2}}{2(\delta_1^{3/2}+\delta_2^{3/2})} \mbox{ and }\ H_0 = - \frac {\delta_1\delta_2}{\sqrt{\pi}(\delta_1^{3/2}+\delta_2^{3/2})}
$$
in the last two formulas.
\qed
\begin{lemma}
\label{le:estimate from below} 
There exists  a constant $c_0 > 0$ satisfying
$$
c_0 \le u < 1\ \mbox{ on }\ \partial\Omega \times (0,1].
$$
\end{lemma}

\noindent
{\it Proof. }  First of all, by the strong maximum principle 
\begin{equation}
\label{between 0 and 1}
0 < u < 1 \ \mbox{ in }\ \mathbb R^N\times(0,+\infty).
\end{equation}
\par
Consider the signed distance function $d^*= d^*(x)$ of $x \in \mathbb R^N$ to the boundary $\partial\Omega$ defined by
\begin{equation}
\label{signed distance}
d^*(x) = \left\{\begin{array}{rll}
 \dist(x,\partial\Omega)\ &\mbox{ if }\ x \in \Omega,
\\
-\dist(x,\partial\Omega)\ &\mbox{ if  }\ x \not\in \Omega.
\end{array}\right.
\end{equation}
 Since $\partial\Omega$ is $C^2$ and compact, there exists a number $\rho > 0$ such that $d^*(x)$ is $C^2$-smooth on a compact neighborhood  $\mathcal N$ of the boundary $\partial\Omega$ given by
$$
\mathcal N = \{ x \in \mathbb R^N : -\rho \le d^*(x) \le \rho \}.
$$
\par
We set now
\begin{equation}
\label{subsolution-Cauchy}
w(x,t) = f\left(t^{-\frac12} d^*(x)\right)\ \mbox{ for }\  (x,t) \in \mathbb R^N \times (0,+\infty).
\end{equation}
Then, it follows from a straightforward computation and the properties of $f$ that 
$$
w_t - \Delta\phi(w) =-\frac 1t f^\prime\left(t^{-\frac12} d^*\right)\left\{-\delta+\sqrt{t}\ \phi^\prime(f)\Delta d^*\right\}\ \mbox{ in }\ \mathcal N\times(0,+\infty).
$$
Notice that if $\sqrt{t} < \frac \delta{\delta_2\max_{\mathcal N}|\Delta d^*|}$, then $w_t - \Delta\phi(w) < 0$.
Hence, since $\partial\mathcal N$ is compact, in view of Lemma \ref{le: a boundary value problem for ODE}, (\ref{subsolution-Cauchy}), and (\ref{Cauchy}), 
we observe that there exists a small $\tau > 0$ satisfying
\begin{eqnarray*}
&w_t - \Delta\phi(w) < 0 = u_t-\Delta\phi(u)\ &\mbox{ in }\ \mathcal N \times (0,\tau],
\\
&w \le u \ &\mbox{ on }\ \partial \mathcal N \times(0,\tau],
\\
& w \le u\ &\mbox{ on } \mathcal N \times\{0\}.
\end{eqnarray*}
\par
Therefore it follows from  the comparison principle that $w \le u$ in $\mathcal N \times (0,\tau]$. In particular, we have 
$$
u \ge f(0)\ (>0)\ \mbox{ on }\partial\Omega\times(0,\tau].
$$
Combining this with (\ref{between 0 and 1}) completes the proof.
\qed
\vskip.2cm
\begin{proofE}
Consider  the unique bounded solutions $u^\pm=u^\pm(x,t)$ of the following initial-boundary value problems:
\begin{eqnarray*}
&u^\pm_t =\Delta \phi(u^\pm)\quad&\mbox{ in }\ \Omega\times (0,+\infty),\label{nonlinear-eq-pm}\\
&u^+= 1\ \mbox{ and }\ u^-=c_0 \quad\quad\ \ \ \ &\mbox{ on }\ \partial\Omega\times
(0,+\infty),\label{nonlinear-Dirichlet-pm}\\
&u^\pm=0 \quad\quad\ \ \ \ & \mbox{ on }\ \Omega\times
\{0\}.\label{nonlinear-initial-pm}
\end{eqnarray*}
Then it follows from Lemma \ref{le:estimate from below} and the comparison principle that
\begin{equation}
\label{upper and lower solutions}
u^- \le u \le u^+\ \mbox{ in }\ \Omega\times (0,1].
\end{equation}
\par
By applying Theorem \ref{th:varadhan-nonlinear-2} to $u^\pm$, we have that, as $t\to 0^+,$ both functions
$-4t\Phi(u^\pm(x,t))$ converge to the function $d(x)^2$ uniformly
on every compact set in $\Omega$. 
\par 
On the other hand, since $\Phi$ is increasing in $s >0$, we have from (\ref{upper and lower solutions}) that
$$
-4t \Phi(u^-) \ge -4t\Phi(u) \ge -4t\Phi(u^+)\ \mbox{ in }\ \Omega\times (0,1],
$$
which implies that, as $t\to 0^+,$ the function
$-4t\Phi(u(x,t))$ converges to the function $d(x)^2$ uniformly
on every compact set in $\Omega$. 
This means that  Theorem \ref{th:varadhan-nonlinear}  also holds for the Cauchy problem (\ref{Cauchy}). 
\par
Finally, proceeding as in Section \ref{section2}, 
with the aid of the strong comparison principle for the Cauchy problem,  
we can easily show that Theorems \ref{th:matzoh-nonlinear} and \ref{th:matzoh-nonlinear-unbounded} also hold for the Cauchy problem (\ref{Cauchy}). 
\end{proofE}


\setcounter{equation}{0}
\setcounter{theorem}{0}

\section{Sphere $\mathbb S^N$ and hyperbolic space $\mathbb H^N$}
\label{section5}

The purpose of this section is to show that similar results  hold also for the heat flow in the sphere $\mathbb S^N$ and the hyperbolic space $\mathbb H^N$ with $N \ge 2$. In order to handle $\mathbb S^N$ and $\mathbb H^N$ together, let us put $\mathbb M = \mathbb S^N$ or $\mathbb M = \mathbb H^N$.
\par
Let $\Omega$ be a domain in $\mathbb M$ with bounded $C^2$-smooth boundary $\partial\Omega$, and denote by $L$  the Laplace-Beltrami operator on $\mathbb M$. Let $u=u(x,t)$ be 
the unique bounded solution either of the following initial-boundary value problem for the heat flow:
\begin{eqnarray}
&u_t=L u\ \ &\mbox{in }\ \Omega\times (0,+\infty),\label{heat flow}\\
&u=1\ \ &\mbox{on }\ \pa\Omega\times (0,+\infty),\label{dirichlet heat flow}\\
&u=0\ \ &\mbox{on }\ \Omega\times \{0\},\label{initial heat flow}
\end{eqnarray}
or of the following Cauchy problem for the heat flow:
\begin{equation}
\label{Cauchy S and H}
u_t = L u\ \mbox{ in } \mathbb M  \times (0,+\infty),\ \mbox{ and }\ u = \chi_{\mathbb M\setminus\Omega}\ \mbox{ on }\mathbb M \times\{0\},
\end{equation}
where $ \chi_{\mathbb M\setminus\Omega}$ denotes the characteristic function of the set $\mathbb M \setminus\Omega$.

Denote by $d(x) = \inf\{ d(x,y) : y\in\partial\Omega\}$ the geodesic distance between $x$ and $\partial\Omega$, where $d(x,y)$ is the geodesic distance between two points $x$ and $y$ in $\mathbb M$.
Then, with the aid of a result of Norris \cite[Theorem 1.1, p. 82]{N} concerning the short-time asymptotics of the heat kernel of Riemannian manifolds, we later prove
\begin{theorem}
\label{th:varadhan for Riemannian manifolds}
Let $u$ be the solution either of problem {\rm(\ref{heat flow})-(\ref{initial heat flow})} or of problem {\rm (\ref{Cauchy S and H})} 
in $\mathbb S^N$ or  $\mathbb H^N$. Then the function $-4t\log u(x,t)$ converges  to the function  $d(x)^2$ as $t \to 0^+$ uniformly on every compact set in $\overline{\Omega}$. 
\end{theorem}

Theorem \ref{th:varadhan for Riemannian manifolds} yields the following symmetry results.
\begin{theorem}
\label{th:S and H} 
Let $u$ be the solution either of problem {\rm(\ref{heat flow})-(\ref{initial heat flow})} or of problem {\rm (\ref{Cauchy S and H})} 
in $\mathbb S^N$ or $\mathbb H^N.$ 
In the case of $\mathbb S^N,$ assume that
$\overline{\Omega}$  is contained in a hemisphere in $\mathbb S^N.$ 
\par
Then Theorems  {\rm\ref{th:matzoh-nonlinear}} and {\rm\ref{th:matzoh-nonlinear-unbounded}} hold for $\mathbb H^N$ 
and Theorem {\rm\ref{th:matzoh-nonlinear}}  holds for $\mathbb S^N$ in the sense
that $\partial\Omega$ must be a geodesic sphere in $\mathbb H^N$ or  $\mathbb S^N,$ respectively.
\end{theorem}

\begin{proof}
By Theorem \ref{th:varadhan for Riemannian manifolds}, each stationary isothermic surface of class $C^2$ in $\Omega$ is then parallel to a connected component of 
$\partial\Omega$ in the sense of the geodesic distance.
The claims of our theorem can thus be proved by replacing the method of moving planes, used in Section \ref{section2}
for $\mathbb R^N$, by a straightforward adaptation of the method of moving closed and totally geodesic hypersurfaces for 
$\mathbb S^N$ or $\mathbb H^N$ developed by Kumaresan and Prajapat in \cite{KuP}. 
\end{proof}

\vskip.2cm
\begin{proofF}  Consider the signed distance function $d^* = d^*(x)$ of  $x \in \mathbb M$ given by the same definition as (\ref{signed distance}). Since $\partial\Omega$ is $C^2$ and compact, there exists a number $\rho_0 > 0$ such that $d^*(x)$ is $C^2$-smooth on a compact neighborhood $\mathcal N$ of $\partial\Omega$ given by
$$
\mathcal N = \{ x \in \mathbb M : -\rho_0 \le d^*(x) \le \rho_0 \}.
$$
Set
$$
\mathcal N^- = \{ x \in \mathbb M : -\rho_0 \le d^*(x) \le 0 \} \ \left( \subset \mathcal N \right).
$$
Let $u^- = u^-(x,t)$ be the unique bounded solution of the following Cauchy problem:
\begin{equation}
\label{Cauchy in M-}
u^-_t = L u^-\ \mbox{ in } \mathbb M \times (0,+\infty),\ \mbox{ and }\ u^- = \chi_{\mathcal N^-}\ \mbox{ on }\mathbb M \times\{0\},
\end{equation}
where $ \chi_{\mathcal N^-}$ denotes the characteristic function of the set $\mathcal N^-.$ 
Moreover, for each $0 < \rho < \rho_0$, we set
$$
\mathcal N_\rho = \{ x \in \mathbb M : -\rho \le d^*(x) \le \rho \} \ \left( \subset \mathcal N \right).
$$
For each $0 < \rho < \rho_0$, let $u^{\rho+} = u^{\rho+}(x,t)$ be the unique bounded solution of the following Cauchy problem:
\begin{equation}
\label{Cauchy in M+}
u^{\rho+}_t = L u^{\rho+}\ \mbox{ in } \mathbb M \times (0,+\infty),\ \mbox{ and }\ u^{\rho+} = 2\chi_{\mathcal N_\rho}\ \mbox{ on }\mathbb M \times\{0\},
\end{equation}
where $ \chi_{\mathcal N_\rho}$ denotes the characteristic function of the set $\mathcal N_\rho.$

Let $\rho \in (0, \min\{ 1,\rho_0\} )$ be arbitrarily small.
Then, there exists a number $t_\rho > 0$ satisfying
\begin{equation}
\label{greater than 1}
u^{\rho+} > 1\ \mbox{ in }\ \partial\Omega \times (0, t_\rho].
\end{equation}
Thus it follows from (\ref{greater than 1}) and the comparison principle that for any $(x,t) \in \Omega \times (0,t_\rho]$
\begin{equation}
\label{estimates +-}
\int_{\mathcal N^-} p(t,x,y)\ dy = u^-(x,t) \le u(x,t) \le u^{\rho+}(x,t) = 2\int_{ \mathcal N_\rho} p(t,x,y)\ dy,
\end{equation}
where $p= p(t, x, y)$ denotes the heat kernel or the fundamental solution of the heat equation on the whole $\mathbb M$.

On the other hand,  it follows from  a result of Norris \cite[Theorem 1.1, p. 82]{N} that the function $-4t\log p(t,x,y)$ converges to the function $d(x,y)^2$ as $t \to 0^+$ uniformly on every compact set in $\mathbb M \times \mathbb M$.  
If $\mathcal K$ is any compact set contained in $\overline\Omega,$ there exists a number $t_{\rho,1} \in (0,t_\rho]$ satisfying
$$
\left| -4t\log p(t,x,y) - d(x,y)^2\right| < \rho^2\ \mbox{ for any }\ (t,x,y) \in (0,t_{\rho,1}] \times \mathcal K \times \mathcal N.
$$
Then we have that for any  $(t,x,y) \in (0,t_{\rho,1}] \times \mathcal K \times \mathcal N$
\begin{equation}
\label{estimates of heat kernel}
\exp\left(-\frac {d(x,y)^2 + \rho^2}{4t} \right) \le p(t,x,y) \le \exp\left(-\frac {d(x,y)^2 - \rho^2}{4t}\right).
\end{equation}
\par
Set $m = \max_{x\in \mathcal K} d(x)$. In view of  (\ref{estimates +-}) and  (\ref{estimates of heat kernel}), let us estimate $u$ from above. Since $d(x,y) \ge \max\{ 0, d(x)-\rho \}$ for any $(x,y) \in \mathcal K \times \mathcal N_\rho$, we observe that
$$
d(x,y)^2 -\rho^2 \ge d(x)^2-2m\rho\ \mbox{ for any }\ (x,y) \in \mathcal K \times \mathcal N_\rho.
$$
Combining this inequality and (\ref{estimates of heat kernel}) with (\ref{estimates +-}) yields that for any $(x,t) \in \mathcal K \times (0,t_{\rho,1}]$
\begin{equation}
\label{above}
u(x,t) \le 2\exp\left(-\frac {d(x)^2 - 2m\rho}{4t}\right)\left|\mathcal N_\rho\right|,
\end{equation}
where $\left|\mathcal N_\rho\right|$ denotes the volume of $\mathcal N_\rho$.  

Next, we proceed to estimating $u$ from below. For each $x \in \mathcal K$, there exists a point $z \in \partial\Omega$ with $d(x) = d(x,z)$. Then there exists a geodesic ball $B$ with radius $\frac 12\rho$ satisfying
$$
B \subset \mathcal N^- \ \mbox{ and }\ \overline{B} \cap \overline{\Omega} = \{ z \}.
$$
Since $d(x,y) \le d(x,z) + d(z,y) \le d(x) + \rho$ for any $y \in B$, we have
$$
d(x,y)^2 + \rho^2 \le d(x)^2 + 2(m+1)\rho\ \mbox{ for any }\ y \in B.
$$
Combining this inequality and (\ref{estimates of heat kernel}) with (\ref{estimates +-}) yields that for any $(x,t) \in \mathcal K \times (0,t_{\rho,1}]$
\begin{equation}
\label{below}
u(x,t) \ge  \exp\left(-\frac {d(x)^2 + 2(m+1)\rho}{4t}\right)\left| B\right|,
\end{equation}
where $\left| B\right|$ denotes the volume of $B$ and $|B|$ is independent of $x \in \mathcal K$.  

Therefore, it follows from (\ref{above}) and (\ref{below}) that there exists a number $t_{\rho,2} \in (0,t_{\rho,1}]$ satisfying for any $(x,t) \in \mathcal K \times (0,t_{\rho,2}]$
\begin{equation}
\label{last estimates}
d(x)^2 -(2m+1)\rho \le -4t\log u(x,t) \le d(x)^2 + (2m+3)\rho.
\end{equation}
This completes the proof.
\end{proofF}

\appendix

\renewcommand{\theequation}{\Alph{section}.\arabic{equation}}
\setcounter{equation}{0}
\setcounter{theorem}{0}

\section{Proof of inequality \eqref{main-inequality}}
\label{section6}

Let us write $v = v^\varepsilon$ for simplicity. By (\ref{at-maximum}) we have at $(x_0,t_0)$
\begin{eqnarray}
&& z_t= 2\zeta\zeta_t|\nabla v|^2 + 2\zeta^2 v_{x_k} v_{x_k t}-\lambda v_t \ge 0, \label{t-derivative}
\\
&& z_{x_i}=2\zeta\zeta_{x_i}|\nabla v|^2 + 2\zeta^2v_{x_k}v_{x_kx_i}-\lambda v_{x_i} = 0,\label{x-derivative}
\\
&& \Delta z = 2 |\nabla\zeta|^2|\nabla v|^2 + 2\zeta\Delta \zeta|\nabla v|^2 + 8\zeta\zeta_{x_i} v_{x_k} v_{x_kx_i}+ 2\zeta^2 v_{x_kx_i}v_{x_kx_i}\nonumber
\\
&&\qquad\quad + 2 \zeta^2 v_{x_k}(\delta v)_{x_k} -\lambda \Delta v \le 0,\label{Laplacian z}
\end{eqnarray}
where the summation convention is understood. Hence, it follows from (\ref{t-derivative}) and (\ref{Laplacian z}) that
\begin{eqnarray*}
0 &\le& z_t- \varepsilon \phi^\prime \Delta z
\\
&=& -\lambda(v_t-\varepsilon\phi^\prime\Delta v) + 2\zeta^2 v_{x_k}(v_t-\varepsilon\phi^\prime\Delta v)_{x_k}
 + 2\zeta^2 v_{x_k}\varepsilon\phi^{\prime\prime} \Psi^\prime (-\varepsilon^{-1})v_{x_k}\Delta v 
\\
&&+ 2\zeta\zeta_t|\nabla v|^2 
\\
&&-\varepsilon\phi^\prime\left\{ 2|\nabla \zeta |^2|\nabla v|^2+ 2\zeta\Delta\zeta|\nabla v|^2 + 8\zeta\zeta_{x_i} v_{x_k}v_{x_kx_i}+ 2\zeta^2v_{x_kx_i}v_{x_kx_i}\right\}.
\end{eqnarray*}
Then with the aid of (\ref{pressure-equation}), we get
\begin{eqnarray*}
0 &\le& \lambda |\nabla v|^2 + 2\zeta^2 v_{x_k}(-|\nabla v|^2)_{x_k} - 2\zeta^2|\nabla v|^2\phi^{\prime\prime} \Psi^\prime\Delta v + 2\zeta\zeta_t|\nabla v|^2
\\
&& -\varepsilon\phi^\prime\left\{ 2|\nabla \zeta |^2|\nabla v|^2+ 2\zeta\Delta\zeta|\nabla v|^2 + 8\zeta\zeta_{x_i} v_{x_k}v_{x_kx_i}+ 2\zeta^2v_{x_kx_i}v_{x_kx_i}\right\}. 
\end{eqnarray*} 
Here, we use $0=z_{x_i}v_{x_i}$ with the aid of (\ref{x-derivative}) to obtain
\begin{eqnarray*}
2\zeta^2 v_{x_k}(-|\nabla v|^2)_{x_k} &=& -4\zeta^2 v_{x_k}v_{x_i}v_{x_kx_i}
\\
&=& -4\left\{\frac\lambda 2 |\nabla v|^2 - \zeta\zeta_{x_i}v_{x_i}|\nabla v|^2\right\}
\\
&=& -2\lambda|\nabla v|^2 + 4 \zeta\zeta_{x_i}v_{x_i}|\nabla v|^2.
\end{eqnarray*}
Then it follows that
\begin{multline*}
\lambda|\nabla v|^2 \le  4 \zeta\zeta_{x_i}v_{x_i}|\nabla v|^2- 2\zeta^2|\nabla v|^2\phi^{\prime\prime} \Psi^\prime\Delta v + 2\zeta\zeta_t|\nabla v|^2
\\
-\varepsilon\phi^\prime\left\{ 2|\nabla \zeta |^2|\nabla v|^2+ 2\zeta\Delta\zeta|\nabla v|^2 + 8\zeta\zeta_{x_i} v_{x_k}v_{x_kx_i}+ 2\zeta^2v_{x_kx_i}v_{x_kx_i}\right\}. 
\end{multline*} 
Here,  by using inequality $ ab \le \frac 18 a^2+2b^2$, we get
\begin{eqnarray*}
-8\varepsilon\phi^\prime\zeta\zeta_{x_i} v_{x_k}v_{x_kx_i} &\le& \varepsilon\phi^\prime\left\{ \zeta^2|\nabla^2v|^2 + 16|\nabla\zeta|^2|\nabla v|^2\right\}
\\
&\le& \varepsilon\phi^\prime\zeta^2|\nabla^2v|^2 + 16\varepsilon\delta_2|\nabla\zeta|^2|\nabla v|^2.
\end{eqnarray*}
Consequently, combining these inequalities with Lemma \ref{le:uniform-bound} yields inequality (\ref{main-inequality}).

\end{document}